\documentclass[reqno]{amsart}
\usepackage{amssymb}
\usepackage{amsmath}
\usepackage{amsthm}
\usepackage{amsfonts}
\usepackage{verbatim}	
\usepackage{mathabx}	
\usepackage{pifont}		
\usepackage{wasysym}	
\usepackage{romanbar}	
\usepackage{color}		

\def\CONST{\operatorname{const}}

\def\DEF{\stackrel{\rm def}{=}}

\def\({\left(}			\def\){\DOTSX\right)}

\def\lv{\left|}			\def\rv{\DOTSX\right|}
\def\lb{\left\{}		
		
\def\<{\left\langle}	\def\>{\right\rangle}

\def\NN{\mathbb{N}}

\def\RR{\mathbb{R}}

\newcommand{\bU}{\mathbf{U}}
\newcommand{\bV}{\mathbf{V}}
\newcommand{\bX}{\mathbf{X}}

\newcommand{\bZ}{\mathbf{Z}}

\newcommand{\bA}{\mathbf{A}}
\newcommand{\bB}{\mathbf{B}}
\newcommand{\balpha}{\bA}	
\newcommand{\bbeta}{\bB}	
\newcommand{\bOmega}{\mathbf{\Omega}}

\newcommand{\cO}{\mathcal{O}}

\begin{document}

\parindent=0pt\parskip=8pt
\linespread{1.3}

\numberwithin{equation}{section}

\newtheorem{theorem}{Theorem}[section]
\newtheorem{lemma}[theorem]{Lemma}
\newtheorem{corollary}[theorem]{Corollary}
\newtheorem{proposition}[theorem]{Proposition}
\newtheorem{example}[theorem]{Example}

\newtheorem*{acknowledgments}{Acknowledgments}

\theoremstyle{definition}
\newtheorem{note}[theorem]{Note}

\title[A family of nonlinear difference equations]
{A family of nonlinear difference equations: \\ 
existence, uniqueness, and asymptotic behavior of positive solutions}


\author[S.~M.~Alsulami]{Saud M. Alsulami}
\address{SMA: KAU, Jeddah, Saudi Arabia}
\email{alsulami@kau.edu.sa}

\author[P.~Nevai]{Paul Nevai}
\address{PN: KAU, Jeddah, Saudi Arabia, \texttt{and} Upper Arlington (Columbus), Ohio, USA}
\email{paul@nevai.us}

\author[J.~Szabados]{J\'ozsef Szabados}
\address{JSZ: Alfr\'ed R\'enyi Institute of Mathematics, Hungarian Academy of Sciences, Budapest, Hungary}
\email{szabados.jozsef@renyi.mta.hu}

\author[W.~Van Assche]{Walter Van Assche}
\address{WVA: KU Leuven, Belgium}
\email{walter.vanassche@wis.kuleuven.be}

\thanks{
The research of Saud M. Alsulami and Paul Nevai was supported by KAU grant No. 20-130/1433 HiCi.
The research of Walter Van Assche was supported by KU Leuven research grant
OT/12/073 and FWO research grant G.0934.13} 

\begin{abstract}

We study solutions $\(x_n\)_{n\in\NN}$ of nonhomogeneous nonlinear
second order difference equations of the type
\begin{align} \notag 
	\ell_n = 
		x_n &\( \sigma_{n,1} \, x_{n+1} + \sigma_{n,0} \, x_n + \sigma_{n,-1} \, x_{n-1} \) + \kappa_n \, x_n ,
		\quad n\in\NN , 
	\\
	&\mbox{with given initial data } 
		\{ x_0 \in \RR \;\; {\rm{\&}} \;\; x_1 \in \RR^+ \} \notag
\end{align}
where
\begin{equation} \notag 
	\(\ell_n\)_{n\in\NN} \in \RR^+
	\; {\rm{\&}} \;
	\(\sigma_{n,0}\)_{n\in\NN} \in\RR^+
	\; {\rm{\&}} \;
	\(\kappa_n\)_{n\in\NN} \in\RR ,
\end{equation}
and the left and right $\sigma$-coefficients satisfy either
\begin{equation} \notag 
	\(\sigma_{n,1}\)_{n\in\NN} \in\RR^+
	\; {\rm{\&}} \;
	\(\sigma_{n,-1}\)_{n\in\NN} \in\RR^+
\end{equation}
or
\begin{equation} \notag 
	\(\sigma_{n,1}\)_{n\in\NN} \in\RR^+_0
	\; {\rm{\&}} \;
	\(\sigma_{n,-1}\)_{n\in\NN} \in\RR^+_0  .
\end{equation}
Depending on one's standpoint, such equations originate either from
orthogonal polynomials associated with certain Shohat-Freud-type
exponential weight functions or from Painlev\'e's discrete equation \#1,
that is, $\textup{d-P}_{\mbox{\small \Romanbar{I}}}$.

\end{abstract}

\dedicatory{Dedicated to Dick Askey}


\keywords{nonhomogeneous nonlinear second order difference equations,
Shohat-Freud-type exponential weight functions, Painlev\'e's discrete
equation \#1, existence of solutions, unicity of solutions, asymptotic
behavior}


\date{submission: November 30, 2013; revision: March 25, 2014;
accepted: April 14, 2014; available online: MMM DD, 2014}

\maketitle

\section{Preliminaries}
\label{sec_prelim}

Since the authors come from different cultures using different
mathematical notation, we need to fix some of it right now in order to
avoid subsequent misunderstanding.

The set of natural numbers $\NN$ consists of all strictly positive
integers. Furthermore, $\RR^+ \DEF \{x\in\RR: x > 0\}$ and $\RR^+_0 \DEF
\{x\in\RR: x \ge 0\}$.


\section{Introduction}
\label{sec_intro}

This section will explain how the unlikely pair of JSZ and WVA became
involved in this research via PN's manipulations. We justify its unusual
length compared to the rest of the paper by the necessity of giving a proper
historical perspective that will also serve as introduction for the
subsequent papers that we plan to publish on nonlinear difference equations.

It was G\'eza Freud who brought the attention of the approximation theory and
orthogonal polynomial communities to exponential weight functions with his
extensive body of work in the 1970s that was suddenly interrupted by his
untimely death in 1979 at the youthful age of 57 years.\footnote{This
statement is not entirely accurate; e.g., Mkhitar Djrbashian (aka
Dzhrbashjan \& Jerbashian) has a large body of work that is not that different
from some of Freud's work but its international impact was negligible.
In addition, exponential weights have long been of definite interest in areas such
as the moment problem.} In particular, Freud solved two special and, to some
extent, simple cases of his \emph{Freud conjectures} that, even today, are of
extraordinary interest despite having been overshadowed by the incomparably
deeper pathbreaking achievements by so many of us such as Alphonse Magnus,
Evguenii A. Rakhmanov, Andrei A. Gonchar, Hrushikesh N. Mhaskar, Edward B.
Saff, Doron S. Lubinsky, Vilmos Totik, and Guillermo L\'opez Lagomasino, in
some kind of a chronological order.

The two special cases above refer to the asymptotic behavior of the
recurrence coefficients in the three-term recurrence for the orthogonal
polynomials associated with the weight functions $|x|^\rho\exp(-x^4)$
and $|x|^\rho\exp(-x^6)$ on $\RR$ with $\rho>-1$, see \cite{freud1976}.
In particular, \cite[(23, p.~5]{freud1976} is the almost-birthplace of
the equation
\begin{equation} \label{eq_xto4withrho}
	n+ \frac {1-(-1)^n}{2} \rho = 4 a_n^2 \(a_{n+1}^2 + a_n^2 + a_{n-1}^2\),
	\qquad a_0=0, \quad n\in\NN ,
\end{equation}
where $\(a_n\)_{n\in\NN}$ are the recurrence coefficients for the
orthogonal polynomials associated with $|x|^\rho\exp(-x^4)$. We wrote
``almost-birthplace'', because, as it was discovered in 1983 by Dick
Askey, see \cite[p.~285]{nevai1983}, Shohat in 1939 could have found
(\ref{eq_xto4withrho}) except that he was only interested in the weight
function $\exp(-x^4)$, that is, when $\rho=0$, see \cite[(39),
p.~407]{shohat1939}. Even if Shohat found or could have found the
equation, he did nothing with it and neither did Freud except that Freud
used a clever $\liminf$--$\limsup$ argument, we call it the \emph{Freud
Kunstgriff\/}, to find the asymptotic behavior of $\(a_n\)$'s in
(\ref{eq_xto4withrho}), see \cite[part~(b),
p.~5]{freud1976}.\footnote{It was subsequently successfully adapted by
several authors, see, \cite[Theorem~1, p.~266]{nevai1983},
\cite[Theorem~2.2(b), pp.~210--211]{hajela1987}, and
\cite[p.~695]{vanassche2007}.} Let us emphasize that for both Shohat and
Freud the $\(a_n\)$'s were recurrence coefficients for the orthogonal
polynomials although the equation itself could have been viewed
independently of  orthogonal polynomials with the stipulation that the
$\(a_n\)$'s are positive.\footnote{The $\(a_n\)$'s appear squared in
(\ref{eq_xto4withrho}) so one could also simply require that they be
real and nonzero.}

PN's 1983 paper \cite{nevai1983} was the first one to subject
\begin{equation} \label{eq_xto4}
	n = 4 a_n^2 \(a_{n+1}^2 + a_n^2 + a_{n-1}^2\),
	\qquad a_0=0, \quad n\in\NN ,
\end{equation}
to a systematic analysis and it was the almost-birthplace of the
theorem that is the starting point of the current paper, see
\cite[Theorem~3, p.~268]{nevai1983}.

\begin{theorem}\label{thm_nevai1983}

The equation
\begin{equation} \label{eq_x_n}
	n = x_n \(x_{n+1} + x_n + x_{n-1}\),
	\qquad x_0=0, \quad n\in\NN ,
\end{equation}
has a unique positive\footnote{In \cite[Theorem~3, p.~268]{nevai1983}
the word ``nonnegative'' is used erroneously, PN's maxima culpa.}
solution $\(x_n\)_{n\in\NN}$, and this solution is obtained by setting
\begin{equation} \label{eq_x_1} \notag
	x_1 =
	\frac
	{\int_\RR x^2 \exp \( - x^4 \) }
	{\int_\RR \exp \( - x^4 \) }
	\, dx
	=
	2 \, \frac {\Gamma\(3/4\)} {\Gamma\(1/4\)} \, .
\end{equation}

\end{theorem}

We wrote ``almost-birthplace'', because while PN was working on
\cite{nevai1983}, he visited the IBM Research Center in Yorktown
Heights, New York, in December, 1981, where he discussed orthogonal
polynomials with Freud-type exponential weights and mentioned a
conjecture that was the essence of Theorem~\ref{thm_nevai1983}. John S.
Lew was in the audience and one thing led to another. In the end, Lew,
together with Donald~A.~Quarles, wrote a magnificent paper that, as far
as we know, was the first study of generalizations of (\ref{eq_x_n})
where orthogonal polynomials no longer occupied a central place and the
primary object of interest was existence and uniqueness of positive
solutions. Lew--Quarles's equation is
\begin{equation} \label{eq_l_n_and_x_n}
	\ell_n = x_n \(x_{n+1} + x_n + x_{n-1}\),
	\qquad x_0\in\RR, \quad \ell_n > 0, \quad n\in\NN ,
\end{equation}
and they proved a very general theorem that contains, as a special case,
\cite[Theorem~3, p.~268]{nevai1983}, see \cite[Theorem~6.3,
p.~369]{lew1983}.\footnote{Interestingly, PN and Lew--Quarles mutually
cross-credit each other for the result; the reason being that (i) they
corresponded while working on their papers, and (ii) PN happened to be
the editor of Lew--Quarles's paper that was published in
J.~Approximation Theory.}

The year 1984 produced two more papers
\cite{nevai1984_1,nevai1984_2}\footnote{\cite{nevai1984_2} was received
by SIMA on April 5, 1983.} where (\ref{eq_xto4}) is discussed, see
\cite[middle of p.~420]{nevai1984_1} and \cite[(2),
p.~1177]{nevai1984_2}. However, nothing is done with the equation
outside of the scope of orthogonal polynomials.

Real progress came in 1984 with
\cite{nevai1984_3}\footnote{\cite{nevai1984_3} was received by JAT on
March 28, 1983.} where Theorem~\ref{thm_nevai1983} was extended to the
following, see \cite[(iii), p.~142]{nevai1984_3}.

\begin{theorem} \label{thm_nevai1984_3}

Given $c>0$ and $K\in\RR$, the equation
\begin{equation} \label{eq_x_ncK} 
	n = c \, x_n \(x_{n+1} + x_n + x_{n-1}\) + K \, x_n,
	\qquad x_0=0, \quad n\in\NN ,
\end{equation}
has a unique positive solution $\(x_n\)_{n\in\NN}$, and this solution is
obtained by setting
\begin{equation} \label{eq_x_1cK} 
	x_1 =
	\frac
	{\int_\RR x^2 \exp \( - \frac {c} {4} x^4 - \frac {K} {2} x^2 \) }
	{\int_\RR \exp \( - \frac {c} {4} x^4 - \frac {K} {2} x^2 \) }
	\, dx \, .
\end{equation}

\end{theorem}

Of course, the $c$ parameter adds nothing new and in all proofs it can be
assumed, without loss of generality, to be equal to $1$ or $4$ or whatever
one finds more convenient.
However, the additional parameter $K$ represents real progress. It comes
from orthogonal polynomials associated with weights $\exp \( - \frac {c}
{4} x^4 - \frac {K} {2} x^2 \) $ on $\RR$. As a matter of fact,
\cite{nevai1984_3} is the ``almost-birthplace'' of orthogonal
polynomials associated explicitly with such weight functions. We wrote
``almost-birthplace'', because they also appear in Daniel Bessis' 1979
paper \cite[(III.1), p.~151]{bessis1979} where the weight function is
$\exp \( - \beta x^4 - \frac {1} {2} x^2 \) $ on $\RR$. As long as $K >
0$, these two weights are equivalent to each other. However, as soon as
$K < 0$, the rules of the game change drastically. We will return to
this in a moment. For some reason unknown to us, \cite{bessis1979}
doesn't treat the case $K=0$ even though in 1979 that would have been
opening up new vistas as well. The equivalent of (\ref{eq_x_ncK}) is
lurking in \cite[(IV.18), p.~151]{bessis1979} a telescoping summation
leads to the equivalent of (\ref{eq_x_ncK}). In the 1980
Bessis--Itzykson--Zuber paper \cite{bessis1980}, the $K>0$ equivalent
also pops up although it's a little harder to recognize it, see, e.g.,
\cite[(4.32), p.~126]{bessis1980} where it is referred to as the
``quartic case'', and then \cite[(4.33), p.~126]{bessis1980} is the
equation corresponding to (\ref{eq_x_ncK}). Since none of us is capable
of understanding either of these papers, we won't comment on
them except for emphasizing that in both papers $K>0$. If the reader is
interested, he can check out \cite[\S6, p.~128--131]{bessis1980},
especially the last sentence that refers to ``\emph{$N\to\infty$ selects
out a unique initial condition, in the sense of asymptotic series, which
is precisely (6.18)\/}'' where the latter formula is essentially
(\ref{eq_x_1cK}).

For the sake of fairness, let us point out that in the 1980s neither
Lew--Quarles nor PN were familiar with \cite{bessis1979,bessis1980}. Had
they been aware of these papers, it might have been a game changer.

The reason that we mentioned these two papers is that they subsequently
became the standard reference as the birthplace of Painlev\'e's discrete
equation \#1, that is, $\textup{d-P}_{\mbox{\small \Romanbar{I}}}$ even
though the case $K<0$ was not even considered in them. On the other
hand, \cite{nevai1984_3} was fully ignored by practically all Painlev\'e
experts.
If the reader wants to find out what Painlev\'e
$\textup{d-P}_{\mbox{\small \Romanbar{I}}}$ is, he can turn to Google
or, even better, read one of Alphonse Magnus' excellent survey papers
such as \cite{magnus1999} who is also well familiar with the work done
by PN and his collaborators in the 1980s.

In 1984 PN mentioned his papers and those of Stan Bonan and Lew--Quarles  to
Dan Hajela who at the time was a student in his introductory real
analysis class, and told him how interesting it would be to find new
approaches to studying difference equations of the type mentioned above.
Hajela turned his attention to a combination of (\ref{eq_l_n_and_x_n}) and
(\ref{eq_x_ncK}), and in \cite{hajela1987} he came up
with\footnote{Hajela in \cite[Theorem~2.2, p.~210]{hajela1987} writes $\ell_n \ge
0$ but that appears to be a typo.}
\begin{equation} \label{eq_l_n_and_x_n_and_k_n}
	\ell_n = x_n \(x_{n+1} + x_n + x_{n-1}\) + \kappa_n \, x_n,
	\qquad x_0\in\RR, \; \ell_n > 0, \;  \kappa_n \in \RR, \; n\in\NN .
\end{equation}
Among others, he found a new proof of Theorem~\ref{thm_nevai1984_3} but,
very unfortunately, only for the case when $K\in\RR_+$ where $\RR_+$ is,
again very unfortunately, undefined, although clearly it is either the set of
positive or nonnegative real numbers, most likely the latter, see
\cite[Theorem~2.2, p.~210]{hajela1987}.


The proof of Theorem~\ref{thm_nevai1983} is elementary whereas the proof
of Theorem~\ref{thm_nevai1984_3} is anything but elementary. Although PN
was the editor of \cite{hajela1987}, he somehow missed or forgot that
for uniqueness of positive solutions in (\ref{eq_l_n_and_x_n_and_k_n})
the parameter $\kappa_n$ had to be nonnegative. Hence, for 25 years, PN
was under the false impression that there is a proof of
Theorem~\ref{thm_nevai1984_3} that is not based on orthogonal
polynomials, Fourier integrals, and the moment problem, but, instead,
relies on some rather elementary fixed point arguments. Therefore, he no
longer sought an elementary solution although the equation
(\ref{eq_x_ncK}) was always on his mind. As a matter of fact, PN
mentioned (\ref{eq_x_ncK}), and its special case (\ref{eq_x_n}) to
Vilmos Totik March of 2003 who thought it would be a good problem for a
Schweitzer competition,\footnote{See
\texttt{en.wikipedia.org/wiki/Mikl\'os\_Schweitzer\_Competition}.} and the
uniqueness of positive solutions of (\ref{eq_x_n}) was indeed included as
Problem~\#6 in 2003.\footnote{Go to the website
\texttt{versenyvizsga.hu/external/vvszuro/vvszuro.php}, first click on
``\emph{B\"ong\'e\-sz\'es}'', then on ``\emph{Schweitzer Mikl\'os
Eml\'ekverseny}', and scroll down to ``\emph{2003 1. kateg. 1. ford. 13.
\'evfolyam}''.} 
Thanks to the participants and Vilmos Totik, we were given access to some of
the ingenious proofs by Rezs\H{o} L\'aszl\'o Lovas, Andr\'as M\'ath\'e,
Tam\'as Terpai, and P\'eter Varj\'u. Varj\'u even included a proof for the
existence of positive solutions that we borrowed and adopted in this paper,
see Theorem~\ref{theorem_existence1}.\footnote{As one of the referees pointed
it out, \cite{collatz1966} is a good source for fixed point theorems and
monotonically decomposable iterative processes, see especially
\cite[\S21]{collatz1966}.} We thank all of them for sharing their solutions
with us.



Fast forward to February of 2013. PN and JSZ spent two weeks with SA at KAU in
Jeddah chock-full of heated discussions of equations of the type described
above and lamenting that there is a lack of any new developments in the
area of existence and uniqueness of positive solutions. At the end of their
visits they flew to Riyadh to attend a workshop on special functions
where they met WVA who overheard them talking about the above equations
and casually mentioned that his talk next day will be
about discrete Painlev\'e equations which is just a fancy term
describing the same object. The rest is history and this is the first
installment of what is expected to be a long term research project.

\section{Notation}
\label{sec_notation}

For $a\in\RR$, the negative and positive parts of $a$ are denoted by $a^-$
and $a^+$, respectively; they are defined the usual way, for instance, $a^-
\DEF (|a| - a)/2$.

We call a sequence, say, $\bZ \DEF \(z_n\)$ positive, if $\(z_n\) \in \RR^+$,
that is, $z_n > 0$ for each $n$ in the domain of $\bZ$. A sequence $\bZ \DEF
\(z_n\)$ is nonnegative if $\(z_n\) \in \RR^+_0$, that is, $z_n \ge 0$ for
each $n$ in the domain of $\bZ$.

We will study solutions $\bX \DEF \(x_n\)_{n\in\NN}$ of nonhomogeneous nonlinear
second order difference equations (recurrence or recursive formulas)
of the type
\begin{align} \label{eq_curr_x_n}
	\ell_n = 
		x_n &\( \sigma_{n,1} \, x_{n+1} + \sigma_{n,0} \, x_n + \sigma_{n,-1} \, x_{n-1} \) + \kappa_n \, x_n ,
		\quad n\in\NN , 
	\\
	&\mbox{with given initial data } 
		\{ x_0 \in \RR \;\; {\rm{\&}} \;\; x_1 \in \RR^+ \} \notag
\end{align}
where
\begin{equation} \label{eq_cond1}
	\(\ell_n\)_{n\in\NN} \in \RR^+
	\; {\rm{\&}} \;
	\(\sigma_{n,0}\)_{n\in\NN} \in\RR^+
	\; {\rm{\&}} \;
	\(\kappa_n\)_{n\in\NN} \in\RR ,
\end{equation}
and the left and right $\sigma$-coefficients satisfy either
\begin{equation} \label{eq_cond2}
	\(\sigma_{n,1}\)_{n\in\NN} \in\RR^+
	\; {\rm{\&}} \;
	\(\sigma_{n,-1}\)_{n\in\NN} \in\RR^+
\end{equation}
or
\begin{equation} \label{eq_cond3}
	\(\sigma_{n,1}\)_{n\in\NN} \in\RR^+_0
	\; {\rm{\&}} \;
	\(\sigma_{n,-1}\)_{n\in\NN} \in\RR^+_0  .
\end{equation}
Note that in (\ref{eq_cond1}) \& (\ref{eq_cond2}) all $\sigma$-coefficients
must be positive whereas in (\ref{eq_cond1}) \& (\ref{eq_cond3}) the left and
right $\sigma$-coefficients may vanish. However, the biggest semantic
difference between (\ref{eq_cond2}) and (\ref{eq_cond3}) is that in the
latter case, because the coefficient of $x_{n+1}$ may vanish, pedantically
speaking, the terms \emph{recurrence} or \emph{recursive formula} are no
longer appropriate although the term \emph{difference equation} is still
valid.

\section{Existence}
\label{sec_existence}

In this section, we prove the following theorem about existence of positive
solutions of (\ref{eq_curr_x_n}).

\begin{theorem}\label{theorem_existence1}

Let the conditions in {\rm(\ref{eq_cond1})} and {\rm(\ref{eq_cond2})} be
satisfied. Then, for every $x_0 \in \RR$,  there exists at least one $x_1
\in \RR^+$ such that the equation {\rm(\ref{eq_curr_x_n})} has a positive
solution $\bX = \(x_n\)_{n\in\NN}$.

\end{theorem}

\begin{proof}

Introducing $t \DEF x_1$, we may view $x_n \equiv x_n(t)$, for each
$n\in\NN$, as a function of $t$ on $\RR^+_0$ with the exception of those
points $t$ where the equation (\ref{eq_curr_x_n}) can't be solved for the
senior term because either the previous one vanishes or when some of the
earlier terms are undefined. For instance, $x_1(t) \equiv t$. 

First, we will construct two strictly monotone sequences
$\balpha \DEF \(\alpha_n\)_{n\in\NN}$ and $\bbeta \DEF \(\beta_n\)_{n\in\NN}$ such that
\begin{equation} \label{eq_sequences}
	-\infty < \alpha_1 < \alpha_2 < \alpha_3 < \dots < \beta_3 < \beta_2 < \beta_1 < \infty
\end{equation}
with the property that, for each $n\in\NN$, we have $x_n(t) > 0$ for $t \in
(\alpha_n,\beta_n)$ and then (\ref{eq_curr_x_n}) can be solved for
$x_{n+1}(t)$ both in $(\alpha_n,\beta_n)$ and at those endpoints of
$(\alpha_n,\beta_n)$ where $x_n(t) \ne 0$.

We will generate the sequences $\balpha$ and $\bbeta$ in such a way that for
each $n\in\NN$ the function $x_n$ is continuous in $[\alpha_n,\beta_n]$,
\begin{equation} \label{eq_alpha}
	x_n(\alpha_n) =
	\lb
		\begin{array}{ll}
			0 
				& \mbox{if $n\in\NN$ is odd,} 
		\\
			\frac {-\kappa_n + \sqrt{\kappa_n^2 + 4 \sigma_{n,0} \, \ell_n}} {2 \sigma_{n,0}} 
				& \mbox{if $n\in\NN$ is even,}
		\end{array}
	\right.
\end{equation}
and
\begin{equation} \label{eq_beta}
	x_n(\beta_n) =
	\lb
		\begin{array}{ll}
			1 + \frac 
				{- \( \sigma_{1,-1} \, x_0 + \kappa_1 \) + \sqrt{\( \sigma_{1,-1} \, x_0 + \kappa_1 \)^2 + 4 \sigma_{1,0} \, \ell_1}} 
				{2 \sigma_{1,0}} 
				& \mbox{if $n = 1$,}
		\\
			\frac {-\kappa_n + \sqrt{\kappa_n^2 + 4 \sigma_{n,0} \, \ell_n}} {2 \sigma_{n,0}} 
				& \mbox{if $n\in\NN \setminus \{1\}$ is odd,} 
		\\
			0 
				& \mbox{if $n\in\NN$ is even.}
		\end{array}
	\right.
\end{equation}
The construction will be made by semi-complete induction; the reason for treating $n=1$
separately in (\ref{eq_beta}) will be explained shortly. 

We define the first term $\alpha_1 \DEF 0$ and then pick $\beta_1 > 0$ in such a way that
\begin{equation} \label{r11} \notag
	\frac{\ell_1}{\beta_1} - \sigma_{1,0} \, \beta_1 - \sigma_{1,-1} \, x_0-  \kappa_1 < 0 ,
\end{equation}
that is, $\beta_1$ is greater than the positive zero of the quadratic polynomial
\begin{equation} \label{r12} \notag
	\sigma_{1,0} \, t^2 + \( \sigma_{1,-1} \, x_0 + \kappa_1 \) t - \ell_1 ,
\end{equation}
and one possible choice for $\beta_1$ is 
\begin{equation} \label{r13} \notag
	\beta_1 \DEF  
	1 + \frac 
		{- \( \sigma_{1,-1} \, x_0 + \kappa_1 \) + \sqrt{\( \sigma_{1,-1} \, x_0 + \kappa_1 \)^2 + 4 \sigma_{1,0} \, \ell_1}} 
		{2 \sigma_{1,0}} 
		\, .
\end{equation}
We have $x_1(t) \equiv t$ so that $x_1$ is continuous in $[\alpha_1, \beta_1]$, is positive in
$(\alpha_1 , \beta_1)$, and the first relations in (\ref{eq_alpha}) and
(\ref{eq_beta}) are also satisfied with $n = 1$. 

Now let $n = 1$. Then, by (\ref{eq_curr_x_n}) we have 
\begin{equation} \label{r14} \notag
	\sigma_{1,1} \, x_2(t) =
	\frac{\ell_1}{x_1(t)} - \sigma_{1,0} \, x_1(t)- \sigma_{1,-1} \, x_0-\kappa_1 ,
\end{equation}
so that $ x_2$ is continuous on $(\alpha_1, \beta_1)$,
\begin{equation} \label{r15}
	\lim_{t\to\alpha_1+0}x_2(t) = +\infty
\end{equation}
and, by the choice of $\beta_1$ in (\ref{eq_beta}),
\begin{equation} \label{r16} \notag
	\sigma_{1,1} \, x_2(\beta_1) = 
	\frac{\ell_1}{x_1(\beta_1)} - \sigma_{1,0} \, x_1(\beta_1) - \sigma_{1,-1} \, x_0 - \kappa_1 < 0 .
\end{equation}
Therefore, $x_2$ has at least one zero in $(\alpha_1, \beta_1)$. Let
\begin{equation} \label{r17} \notag
	\beta_2 \DEF \inf \{t: t >\alpha_1,\;x_2(t) = 0\} ,
\end{equation}
and then $\alpha_1 < \beta_2 < \beta_1$ and, by continuity, $x_2(\beta_2) = 0$ as well.
Using (\ref{r15}) and the intermediate value theorem, we can find $\alpha_2 \in (\alpha_1, \beta_2)$
such that
\begin{equation} \label{eq_alpha11} \notag
	x_2(\alpha_2) =
	\frac{\sqrt{\kappa_2^2 + 4 \sigma_{2,0} \, \ell_2}-\kappa_2}{2\sigma_{2,0}} \,.
\end{equation}
Hence, $\alpha_2$, $\beta_2$, and $x_2$ all have the prescribed properties. 

The next step involves a semi-complete induction in the sense that passing to
$n+1$ from $n$ will also use the inductive assumption for $n-1$. The
inductive step is almost but not exactly the same as we went to $x_2$ from
$x_1$ and $x_0$.

Suppose the construction with the properties mentioned above has been done up
to $n$, and now we proceed with it for $n+1$. We assume that $n$ is odd; the
other case is similar. Writing (\ref{eq_curr_x_n}) in the form
\begin{equation} \label{r1} \notag
	\sigma_{n,1} \, x_{n+1}(t) =
	\frac{\ell_n}{x_n(t)}-\sigma_{n,0} \, x_n(t)-\sigma_{n,-1} \, x_{n-1}(t) - \kappa_n ,
\end{equation}
we can see that $x_{n+1}$ is continuous on $(\alpha_n,\beta_n)$ because, by
the inductive hypotheses, (i) $x_n$ is positive and continuous on
$(\alpha_n,\beta_n)$, (ii) $x_{n-1}$ is positive and continuous on
$(\alpha_{n-1},\beta_{n-1})$, and (iii) $[\alpha_n,\beta_n] \subset
(\alpha_{n-1},\beta_{n-1})$. Furthermore, by the first relation in
(\ref{eq_alpha}), the second in (\ref{eq_beta}), and by simple algebra, we
have
\begin{equation} \label{eq_values}
	\lim_{t\to\alpha_n+0}x_{n+1}(t) = +\infty
	\quad {\rm{\&}} \quad
	x_{n+1}(\beta_n) =
	- \frac {\sigma_{n,-1}} {\sigma_{n,1}} x_{n-1}(\beta_n) < 0
\end{equation}
because ${\sigma_{n,1}}$ and ${\sigma_{n,-1}}$ are both positive.
Hence, $x_{n+1}$ has at least one zero in $(\alpha_n,\beta_n)$. Let
\begin{equation} \label{eq_beta1} \notag
	\beta_{n+1} \DEF \inf \{t: t >\alpha_n,\;x_{n+1}(t) = 0\} .
\end{equation}
Then, by continuity, $x_{n+1}(\beta_{n+1}) = 0$ and, by the first limit in
(\ref{eq_values}), $\beta_{n+1} \in (\alpha_n,\beta_n)$. By the construction,
$x_{n+1} > 0$ in $(\alpha_n,\beta_{n+1})$. Again by the first limit in
(\ref{eq_values}), the intermediate value theorem guarantees the existence of
$\alpha_{n+1} \in (\alpha_n,\beta_{n+1})$ such that
\begin{equation} \label{eq_alpha1} \notag
	x_{n+1}(\alpha_{n+1}) =
	\frac{\sqrt{\kappa_{n+1}^2+4 \sigma_{n+1,0} \, \ell_{n+1}}-\kappa_{n+1}}{2\sigma_{n+1,0}} \,.
\end{equation}
This proves that $x_{n+1}$ has all the required properties.

Now the theorem follows immediately, since, by the nested interval
theorem,\footnote{We should rather say that by a version of the nested
interval theorem since our intervals are open and we also need the fact that
the closure of each interval lies inside the interior of its parent
interval.} we can pick a number $t^\ast$, that is, an initial value $x_1$
such that
\begin{equation}  \notag 
	\lim_{n\to\infty} \alpha_n \le  t^\ast \le \lim_{n\to\infty}\beta_n ,
\end{equation}
and then the sequence $\{x_n(t^\ast)\}_{n\in\NN}$ is a positive solution of
(\ref{eq_curr_x_n}).
\end{proof}

\begin{note}

Of course, if
\begin{equation} \notag
	\lim_{n\to\infty} \alpha_n < \lim_{n\to\infty}\beta_n
\end{equation}
above, then every $ t^\ast$ between those two limits would yield a
positive solution. However, for all practical purposes this observation
is useless since we have no actual information about those limits. A
similarly ``useless'' observation was made in \cite[Theorem~4.3,
p.~365]{lew1983}, see $x^{\pm}$ there.

\end{note}

\section{Uniqueness}
\label{sec_uniqueness}

In this section, we study uniqueness of positive solutions of
(\ref{eq_curr_x_n}). We will need the following lemma that is no doubt
well known and is straightforward anyway.

\begin{lemma} \label{lemma_convex1}

If $\bOmega \DEF \(\omega_n\)_{n=0}^\infty$ is a convex sequence of real numbers
that grows slower than linear, that is,
\begin{equation} \label{eq_convex1}
	2 \omega_n \le \omega_{n+1} + \omega_{n-1}, \qquad \forall n\in\NN ,
\end{equation}
and
\begin{equation} \label{eq_convex2}
	\liminf_{n\to\infty} \frac {\omega_n} {n} \le 0 \, ,
\end{equation}
then $\bOmega$ is a nonincreasing sequence. In particular, if $\bOmega$
is a nonnegative sequence with $\omega_0 = 0$, then $\omega_n = 0$ for
$n\in\NN$.

\end{lemma}

\begin{proof}

Rewriting (\ref{eq_convex1}) as 
\begin{equation} \notag 
	\omega_n - \omega_{n-1} \le \omega_{n+1} - \omega_n , \qquad n\in\NN ,
\end{equation}
shows that $\(\omega_n - \omega_{n-1}\)_{n\in\NN}$ is a nondecreasing sequence
so that
\begin{equation}  \notag 
	\omega_n - \omega_{n-1} \le \omega_{p} - \omega_{p-1} , 
	\qquad n,p \in \NN \quad {\rm{\&}} \quad  n \le p,
\end{equation}
from which
\begin{align} \notag 
	(q - n) \( \omega_n - \omega_{n-1} \) =
	\sum_{p=n+1}^q \( \omega_n - \omega_{n-1} \) \le 
	\sum_{p=n+1}^q &\( \omega_p - \omega_{p-1} \) = 
	\omega_q - \omega_n , 
	\\
	& n,q \in \NN \quad {\rm{\&}} \quad n < q , \notag
\end{align}
that is,
\begin{equation} \notag 
	\omega_n - \omega_{n-1} \le 
	\frac{\omega_q - \omega_n} {q-n} \, , 
		\qquad n,q \in \NN \quad {\rm{\&}} \quad n < q , \notag
\end{equation}
and now, fixing $n\in\NN$, letting $q\to\infty$, and taking
(\ref{eq_convex2}) into consideration, we finally see that $\omega_n -
\omega_{n-1} \le 0$ for $n\in\NN$, that is, $\bOmega$ is a nonincreasing
sequence.
\end{proof}

We define $\(\sigma_n\)_{n\in\NN}$ by
\begin{equation} \label{eq_sigma}
	\sigma_n \DEF \max\(\sigma_{n,-1},\sigma_{n,1}\), \qquad n\in\NN ,
\end{equation}
see  (\ref{eq_cond3}), so that $\sigma_n \ge 0$ for $n\in\NN$.


\begin{theorem}\label{theorem_uniqueness1}

With the conditions in {\rm(\ref{eq_cond1}) \& (\ref{eq_cond3})} and with the notation
{\rm(\ref{eq_sigma})}, assume
\begin{equation} \label{eq_liminf0}
	\liminf_{n\to\infty}  
	\frac {1} {n^2} 
		\( \frac {\ell_n} {\sigma_{n,0}} + \frac {\(\kappa_n^-\)^2} {\sigma_{n,0}^2} \) = 0 \,.
\end{equation}
Let $\NN$ be representable as the disjoint union $\NN =
\NN_{\ref{eq_uniquenessa}} \cup \NN_{\ref{eq_uniquenessb}}$ such
that, for each $n\in\NN$, one of the following two displayed conditions 
\begin{equation} \label{eq_uniquenessa} \tag{\davidsstar}
	2 \, \sigma_n \le \sigma_{n,0} \, , 
	\qquad \mbox{{\rm if} } n \in \NN_{\ref{eq_uniquenessa}},
\end{equation}
or
\begin{equation} \label{eq_uniquenessb} \tag{\mbox{\ding{67}}}
	\sigma_n \le \sigma_{n,0} < 2 \sigma_n 
	\;\; {\rm{\&}} \;\;
	-2 \( \sigma_{n,0} - \sigma_n \)  \sqrt{\ell_n} \le 
	\kappa_n \sqrt {2 \sigma_n -\sigma_{n,0}} \, ,
	\quad \mbox{{\rm if} } n \in \NN_{\ref{eq_uniquenessb}}
\end{equation}
is satisfied. 
\\ 
In addition, if $1 \in \NN_{\ref{eq_uniquenessa}}$, then simply let $x_0$ in
{\rm(\ref{eq_curr_x_n})} be an arbitrary real number, whereas if $1 \in
\NN_{\ref{eq_uniquenessb}}$, then we also assume that $x_0$  satisfies
\begin{equation} \label{eq_uniquenessb1}
	-2 \( \sigma_{1,0} - \sigma_1 \)  \sqrt{\ell_1} \le 
	\( \sigma_{1,-1} \, x_0  + \kappa_1 \) \sqrt {2 \sigma_1 -\sigma_{1,0}} \, .
\end{equation}
Then there exists a unique $x^\ast>0$ such that if a sequence $\bX =
\(x_n\)_{n\in\NN}$ satisfying equation {\rm(\ref{eq_curr_x_n})} is positive
then $x_1 = x^\ast$, and, therefore, {\rm(\ref{eq_curr_x_n})} can't have
more than one positive solution.

\end{theorem}

Before we prove Theorem~\ref{theorem_uniqueness1}, let us discuss a few special cases.

\begin{example}\label{example_uniqueness1}

In the extreme case when the left and right $\sigma$-coefficients in
{\rm(\ref{eq_cond3})} all vanish, we have $\NN_{\ref{eq_uniquenessa}} =
\NN$ {\rm{\&}} $\NN_{\ref{eq_uniquenessb}} = \emptyset$, and
{\rm(\ref{eq_curr_x_n})} takes the form
\begin{equation}  \notag 
	\ell_n = 
	\sigma_{n,0} \, x_n^2 + \kappa_n \, x_n ,
	\quad n\in\NN ,
\end{equation}
so that, clearly, for each $n\in\NN$, the quadratic equation has a unique
positive solution $x_n$.

\end{example}

\begin{example}\label{example_uniqueness2} 

If we allowed the middle $\sigma$-coefficients in {\rm(\ref{eq_curr_x_n})}
to vanish too, then we could end up with
\begin{equation} \notag 
	\ell_n = 
	\kappa_n \, x_n ,
	\quad n\in\NN ,
\end{equation}
that would not have a positive solution $\bX$ unless
$\(\kappa_n\)_{n\in\NN}$ is positive.

\end{example}

\begin{example}\label{example_uniqueness4}

Other examples showing the significance of the middle $\sigma$-coefficients
in {\rm(\ref{eq_curr_x_n})} are the following. If
\begin{equation} \notag
	1 = x_n (x_{n+1} + x_{n-1}) ,
	\quad n\in\NN ,
\end{equation}
then the substitution $y_n \DEF x_{n-1}x_{n}$ transforms it to
\begin{equation} \notag
	1 = y_{n+1} + y_{n} ,
	\quad n\in\NN ,
\end{equation}
so that $y_n = 1/2 + (-1)^n \CONST$, and if
\begin{equation} \notag
	n = x_n (x_{n+1} + x_{n-1}) ,
	\quad n\in\NN ,
\end{equation}
then the same substitution leads to $y_n = n/2 - 1/4 + (-1)^n
\CONST$.\footnote{These examples were suggested by one of the referees.}

\end{example}

\begin{example}\label{example_uniqueness3}

One can take $\sigma_{1,1}=\sqrt{2}$ {\rm \&} $\sigma_{1,0}=1$ {\rm \&}
$\sigma_{1,-1}=0$, and $\sigma_{n,1}= \sqrt{\frac{n}{n+1}}$ {\rm \&}
$\sigma_{n,0}=1$ {\rm \&} $\sigma_{n,-1} = \sqrt{\frac{n}{n-1}}$ for $n \ge
2$. Then $x_n=\sqrt{n}$ is a solution of
\begin{equation} \notag
	3 n = x_n (\sigma_{n,1} \, x_{n+1} + \sigma_{n,0} \, x_n + \sigma_{n,-1} \, x_{n-1}) ,
	\quad n\in\NN ,
\end{equation}
with $x_0=0$ and $x_1=1$. The solution for $x_0=0$ and $x_1=-1$ is $x_n
= -\sqrt{n}$. These are the only ``nice'' solutions, that is, monotone,
and either positive or negative. The expression $\sqrt{n}$ corresponds
to the asymptotic behavior whenever $\sigma_{n,1}$ and $\sigma_{n,-1}$
converge to $1$ as $n\to\infty$.

\end{example}

Combining Theorems \ref{theorem_existence1} \& \ref{theorem_uniqueness1}
and simplifying the conditions in the latter we get the following
corollary.

\begin{corollary}\label{corollary_uniqueness1}

Let $\sigma_n\le\sigma_{n,0}$ {\rm{\&}} $\kappa_n\ge 0$ for $n\in\NN$, let
\begin{equation} \label{eq_liminf1}
	\liminf_{n\to\infty} \frac{\ell_n}{ n^2 \, \sigma_{n,0}} = 0 \,,
\end{equation}
and assume that either $x_0 = 0$ or, at least, $x_0$ satisfies
$\sigma_{1,-1} x_0  \ge -\kappa_1 $, see {\rm(\ref{eq_cond3})} and
{\rm(\ref{eq_sigma})}. Then there exists a unique $x_1>0$ such that the
sequence $\bX$ satisfying equation {\rm(\ref{eq_curr_x_n})} is positive.

\end{corollary}

\begin{proof}[Proof of Theorem~\ref{theorem_uniqueness1}]

Suppose (\ref{eq_curr_x_n}) has at least one positive solution.
Pick two, not necessarily distinct, positive solutions of
(\ref{eq_curr_x_n}), say, $\bU \DEF \(u_n\)_{n\in\NN}$ and $\bV \DEF
\(v_n\)_{n\in\NN}$. Denoting $\varepsilon_n \DEF u_n-v_n$ for
$n\in\NN\cup\{0\}$ and taking the difference of the corresponding
equations, we obtain
\begin{equation} \label{eq_diff01} 
	\(\sigma_{n,0}+\frac{\ell_n}{u_nv_n}\)\varepsilon_n =
	- \sigma_{n,1}\varepsilon_{n+1} - \sigma_{n,-1}\varepsilon_{n-1} ,
	\qquad n\in\NN ,
\end{equation}
so that, by {\rm(\ref{eq_sigma})},
\begin{equation} \label{eq_diff02} 
	\(\sigma_{n,0}+\frac{\ell_n}{u_nv_n}\) \lv \varepsilon_n \rv \le
	\sigma_n \( \lv \varepsilon_{n+1} \rv + \lv \varepsilon_{n-1} \rv \) ,
	\qquad n\in\NN .
\end{equation}

\textbf{Step 1.} Our first goal is to show that
\begin{equation} \label{eq_epsilon01} 
	2 \lv \varepsilon_n \rv \le \lv \varepsilon_{n+1} \rv + \lv \varepsilon_{n-1} \rv, 
	\qquad \forall n\in\NN .
\end{equation}
If $n \in \NN_{\ref{eq_uniquenessa}}$, then (\ref{eq_epsilon01}) holds
trivially by the theorem's assumption (\ref{eq_uniquenessa}) since all terms
in (\ref{eq_diff02}) are nonnegative.

If $n \in \NN_{\ref{eq_uniquenessb}}$, then we have to consider separately
when $n = 1$ and $n > 1$; the reason being that for $n = 1$ equation
{\rm(\ref{eq_curr_x_n})} includes the term $x_0$ that is not necessarily
nonnegative, and, therefore, $x_0$ can't be thrown away when estimating
$x_1$. 

When $n \in \NN_{\ref{eq_uniquenessb}}$ and $n = 1$, we obtain from
{\rm(\ref{eq_curr_x_n})} the inequality
\begin{equation} \label{eq_curr_x_1_1} \notag
	\ell_1 \ge  
	x_1 \( \sigma_{1,0} \, x_1 + \sigma_{1,-1} \, x_0 \) + \kappa_1 \, x_1 =
	 \sigma_{1,0} \, x_1^2 + \( \sigma_{1,-1} \, x_0 + \kappa_1 \) \, x_1 ,
\end{equation}
so that by (\ref{eq_uniquenessb1})
\begin{equation} \label{eq_curr_x_1_2} \notag
	\ell_1 \ge  
	\sigma_{1,0} \, x_1^2 - 
		\frac {
		2 \( \sigma_{1,0} - \sigma_1 \)  \sqrt{\ell_1}} 
		{\sqrt {2 \sigma_1 -\sigma_{1,0}}} \, x_1 .
\end{equation}
Therefore, $x_1$ must lie between the roots of 
\begin{equation} \label{eq_curr_x_1_3} \notag
	\sigma_{1,0} \, x^2 -
	\frac {2 \( \sigma_{1,0} - \sigma_1 \)  \sqrt{\ell_1}} {\sqrt {2 \sigma_1 -\sigma_{1,0}}} \, x - \ell_1 = 
	0 \, .
\end{equation}
Solving this quadratic equation, we obtain
\begin{equation} \label{eq_curr_x_1_4} 
	x_1 \le \frac {\sqrt{\ell_1}} { \sqrt {2 \sigma_1 -\sigma_{1,0}} } \, ,
	\qquad 1 \in \NN_{\ref{eq_uniquenessb}} .
\end{equation}
Therefore,
\begin{equation} \label{eq_curr_x_1_5} \notag
	 2 \sigma_1 -\sigma_{1,0} \le \frac {\ell_1} { u_1 v_1} \, ,
\end{equation}
which, together with (\ref{eq_diff02}), implies (\ref{eq_epsilon01}) for $n =
1 \in \NN_{\ref{eq_uniquenessb}}$.

When $n \in \NN_{\ref{eq_uniquenessb}}$ and $n > 1$, we proceed the same way with
a minor modification. Namely, we obtain from {\rm(\ref{eq_curr_x_n})} the
inequality
\begin{equation} \label{eq_curr_x_n_1} \notag
	\ell_n \ge  
	x_n \( \sigma_{n,0} \, x_n \) + \kappa_n \, x_n =
	 \sigma_{n,0} \, x_n^2 + \kappa_n \, x_n ,
\end{equation}
so that by the second inequality in (\ref{eq_uniquenessb})
\begin{equation} \label{eq_curr_x_1_6} \notag
	\ell_n \ge  
	\sigma_{1,0} \, x_n^2 - 
	\frac 
	{2 \( \sigma_{n,0} - \sigma_n \)  \sqrt{\ell_n}} 
	{\sqrt {2 \sigma_n -\sigma_{n,0}}} \, x_n ,
\end{equation}
and then the same ``largest root of the quadratic equation'' argument we find that
\begin{equation} \label{eq_curr_x_n_7} \notag
	x_n \le \frac {\sqrt{\ell_n}} { \sqrt {2 \sigma_n -\sigma_{n,0}} } \, ,
	\qquad 1 < n \in \NN_{\ref{eq_uniquenessb}} ,
\end{equation}
that is,
\begin{equation} \label{eq_curr_x_n_8} \notag
	 2 \sigma_n -\sigma_{n,0} \le \frac {\ell_n} { u_n v_n} ,
\end{equation}
which, together with (\ref{eq_diff02}), implies (\ref{eq_epsilon01}) for $n
\in \NN_{\ref{eq_uniquenessb}} \setminus \{1\}$ as well.

\textbf{Step 2.} Our next goal is to show that
\begin{equation} \label{eq_liminf04} 
	\liminf_{n\to\infty} \frac {\lv \varepsilon_n \rv} {n} = 0 .
\end{equation}

We will estimate ${\lv \varepsilon_n \rv}$ for all $n\in\NN\setminus\{1\}$ in one fell swoop.
Throwing away all nonnegative terms in (\ref{eq_curr_x_n}), we obtain
\begin{equation} \label{eq_estimate2} \notag
	\ell_n \ge \sigma_{n,0} \, x_n^2+ \kappa_n \, x_n ,
	\qquad n > 1,
\end{equation}
so that $x_n$ must lie between the roots of $\sigma_{n,0} \, x^2 +
\kappa_n x - \ell_n = 0$, and, therefore,
\begin{equation} \label{eq_estimate3} \notag
	\lv \varepsilon_n \rv \le
	u_n + v_n \le
	2 \times \sqrt{ \frac {\ell_n} {\sigma_{n,0}} }
	\qquad \mbox{{\rm if} } n > 1
	\; {\rm{\&}} \;
	\kappa_n \ge 0,
\end{equation}
and
\begin{equation} \label{eq_estimate4} \notag
	\lv \varepsilon_n \rv \le
	u_n + v_n \le
	\frac {\sqrt{\kappa_n^2 +4 \ell_n \, \sigma_{n,0}}-\kappa_n} {\sigma_{n,0}} \le
	2 \, \times \( \sqrt{ \frac {\ell_n} {\sigma_{n,0}} } - \frac {\kappa_n} {\sigma_{n,0}} \) ,
	\qquad \mbox{{\rm if} } n > 1
	\; {\rm{\&}} \;
	\kappa_n < 0,
\end{equation}
so that
\begin{equation} \label{eq_estimate5} \notag
	\varepsilon_n^2 \le
	8 \times \( \frac {\ell_n} {\sigma_{n,0}} + \frac {\(\kappa_n^-\)^2} {\sigma_{n,0}^2} \) ,
	\qquad \forall n > 1
	\; {\rm{\&}} \;
	\forall \kappa_n \in \RR ,
\end{equation}
and, in view of (\ref{eq_liminf0}), the limit relationship in (\ref{eq_liminf04})
follows.

Combining what was proved in steps 1 \& 2, that is, (\ref{eq_epsilon01}) and
(\ref{eq_liminf04}), we can use now Lemma~\ref{lemma_convex1}, applied with
$\omega_n \DEF \lv \varepsilon_n \rv$, to obtain immediately that the two
solutions $\bU$ and $\bV$ are, in fact, identical.
\end{proof}

\begin{note}

It would be interesting to see either examples or conditions when one
can produce precisely $Q$ different initial data $x_1>0$ yielding
positive solutions where $1<Q\in\NN$ is prescribed.

\end{note}

\begin{note}

Although Lew--Quarles' \cite[Theorem~4.3, p.~365]{lew1983} uses an
entirely different approach to uniqueness, it also assumes that
$\liminf_{n\to\infty} {\ell_n}/{ n^2} = 0$ that is essentially the same
as (\ref{eq_liminf1}). On the other hand, Lew--Quarles'
\cite[Theorem~6.3, p.~369]{lew1983} and Hajela's \cite[Theorem~2.2,
p.~210]{hajela1987} impose the condition $\lim_{n\to\infty}
{\ell_{n+1}}/{\ell_n} > 0$ that is of a totally different nature.

\end{note}

\section{Limits}
\label{sec_limits}

In this section, we investigate asymptotic behavior of (not necessarily
positive or negative) solutions of (\ref{eq_curr_x_n}). As before, we
always assume that $\(\ell_n\)_{n\in\NN} \in \RR^+$. For convenience, we
rewrite (\ref{eq_curr_x_n}), that is,
\begin{equation} \notag
	\ell_n = 
		x_n \( \sigma_{n,1} \, x_{n+1} + \sigma_{n,0} \, x_n + \sigma_{n,-1} \, x_{n-1} \) + \kappa_n \, x_n
\end{equation}
as
\begin{equation} \notag
	 1 = 
		\frac {x_n} {\sqrt{\ell_n}}
		\( 
		\sigma_{n,1} \, \sqrt {\frac {\ell_{n+1}} {\ell_n}} \, \frac {x_{n+1}} {\sqrt{\ell_{n+1}}} + 
		\sigma_{n,0} \, \frac {x_n} {\sqrt{\ell_n}} + 
		\sigma_{n,-1} \, \sqrt {\frac {\ell_{n-1}} {\ell_n}} \, \frac{x_{n-1}} {\sqrt{\ell_{n-1}}}
		\) +
		\frac {\kappa_n} {\sqrt{\ell_n}} \, \frac {x_n} {\sqrt{\ell_n}}
\end{equation}
or, introducing
\begin{equation} \label{x_n_2_t_n}
	t_n \DEF \frac {x_n} {\sqrt{\ell_n}} \, ,
\end{equation}
as
\begin{equation} \label{eq_curr_t_n}
	 1 = 
		t_n
		\( 
		\sigma_{n,1} \, \sqrt {\frac {\ell_{n+1}} {\ell_n}} \, t_{n+1} + 
		\sigma_{n,0} \, t_n + 
		\sigma_{n,-1} \, \sqrt {\frac {\ell_{n-1}} {\ell_n}} \, t_{n-1}
		\) +
		\frac {\kappa_n} {\sqrt{\ell_n}} \, t_n \, .
\end{equation}

We start with the following observation.

\begin{theorem} \label{theorem_bound1}

Let the coefficients of {\rm(\ref{eq_curr_x_n})} satisfy
\begin{equation} \notag
	\liminf_{n\to\infty} \sigma_{n,0}  > 0
	\quad \& \quad
	\limsup_{n\to\infty} \, \sigma_{n,\pm 1} \sqrt { \frac {\ell_{n \pm 1}} {\ell_n} } < \infty
	\quad \& \quad
	\limsup_{n\to\infty} \frac {\lv \kappa_n \rv} {\sqrt{\ell_n}} < \infty \, .
\end{equation}
If $\bX = \(x_n\)_{n\in\NN}$ satisfies {\rm(\ref{eq_curr_x_n})}, then
\begin{equation} \label{liminf_limsup_equiv}
	\liminf_{n\to\infty} \frac {x_n} {\sqrt{\ell_n}} > -\infty
	\quad \Longleftrightarrow \quad
	\limsup_{n\to\infty} \frac {x_n} {\sqrt{\ell_n}} < \infty \, .
\end{equation}

\end{theorem}

\begin{corollary} \label{theorem_bound2}

Let the coefficients of {\rm(\ref{eq_curr_x_n})} satisfy
\begin{equation} \notag
	\limsup_{n\to\infty} \sigma_{n,1}  < \infty
	\quad \& \quad
	\liminf_{n\to\infty} \sigma_{n,0}  > 0
	\quad \& \quad
	\limsup_{n\to\infty} \sigma_{n,-1}  < \infty
\end{equation}
and
\begin{equation} \notag
	0 < 
	\liminf_{n\to\infty} \sqrt { \frac {\ell_{n+1}} {\ell_n} } \le 
	\limsup_{n\to\infty} \sqrt { \frac {\ell_{n+1}} {\ell_n} } <
	\infty
	\quad \& \quad
	\limsup_{n\to\infty} \frac {\lv \kappa_n \rv} {\sqrt{\ell_n}} < \infty .
\end{equation}
If $\bX = \(x_n\)_{n\in\NN}$ satisfies {\rm(\ref{eq_curr_x_n})}, then
{\rm(\ref{liminf_limsup_equiv})} holds.

\end{corollary}

\begin{note}

Of course, we could have stated (\ref{liminf_limsup_equiv}) as
\begin{equation} \notag
	\liminf_{n\to\infty} \frac {x_n} {\sqrt{\ell_n}} > -\infty
	\quad \Longleftrightarrow \quad
	\limsup_{n\to\infty} \frac {x_n} {\sqrt{\ell_n}} < \infty
	\quad \Longleftrightarrow \quad
	\limsup_{n\to\infty} \frac {\lv x_n \rv } {\sqrt{\ell_n}} < \infty \, .
\end{equation}

\end{note}

\begin{example}

If, for instance,  $\ell_n \DEF 1$ \&  $\sigma_{n,1} \DEF 1/\sqrt {n}$
when $n$ is even and $\ell_ n \DEF n$ \& $\sigma_{n,1} \DEF \sqrt {n}$
when $n$ is odd, $\sigma_{n,-1} \DEF 1/\sigma_{n,1}$, and $\sigma_{n,0}
\DEF 1$ \& $\kappa_n \DEF 1$, then the coefficient conditions in
Theorem~\ref{theorem_bound1} are satisfied whereas those in
Corollary~\ref{theorem_bound2} are not.

\end{example}

\begin{proof}[Proof of Theorem~\ref{theorem_bound1}]

First, we will prove $\Longrightarrow$ in (\ref{liminf_limsup_equiv}).
Let $\NN_* \DEF \{n \in \NN : x_n \ge 0 \}$. Clearly, it is sufficient
to estimate $x_n/\sqrt{\ell_n}$ from above only when $n\in\NN_*$. Let $K
\in \RR$ be such that
\begin{equation} \notag
	t_n = \frac {x_n} {\sqrt{\ell_n}} > K , \qquad n \in \NN .
\end{equation}
Then, keeping in mind that $t_n \ge 0$ for $n\in\NN_*$, equation
(\ref{eq_curr_t_n}) implies
\begin{equation} \notag
	1 \ge
	K \, \sigma_{n,1} \, \sqrt {\frac {\ell_{n+1}} {\ell_n}}\, t_n +
	\sigma_{n,0} \, t_n^2 +
	K \sigma_{n,-1} \, \sqrt {\frac {\ell_{n-1}} {\ell_n}} \, t_n +
	\frac {\kappa_n} {\sqrt{\ell_n}} \, t_n \, ,
	\qquad n\in\NN_* ,
\end{equation}
that is, the quadratic polynomial 
\begin{equation} \notag
	\sigma_{n,0} \, t^2 +
	\(
	K \, \sigma_{n,1} \, \sqrt {\frac {\ell_{n+1}} {\ell_n}} +
	K \sigma_{n,-1} \, \sqrt {\frac {\ell_{n-1}} {\ell_n}} +
	\frac {\kappa_n} {\sqrt{\ell_n}}
	\)  \, t - 1 \, ,
	\qquad n\in\NN_* ,
\end{equation}
with positive leading coefficient is nonpositive at $t = t_n$ so that
$t_n$ is at most as big as its largest zero is, that is,
\begin{equation} \notag
	t_n \le \frac {-B_n + \sqrt {B_n^2 + 4 \, A_n}} {2 \, A_n}
 	\, ,
	\qquad n\in\NN_* ,
\end{equation}
where
\begin{equation} \notag
	A_n \DEF \sigma_{n,0}
	\quad \& \quad
	B_n \DEF 
		K \, \sigma_{n,1} \, \sqrt {\frac {\ell_{n+1}} {\ell_n}} +
		K \sigma_{n,-1} \, \sqrt {\frac {\ell_{n-1}} {\ell_n}} +
		\frac {\kappa_n} {\sqrt{\ell_n}} \, .
\end{equation}
Hence, by the conditions imposed on the coefficients of (\ref{eq_curr_x_n}),
$\limsup_{\stackrel {n\to\infty} {n \in \NN_*}} t_n < \infty$ that proves
$\Longrightarrow$ in (\ref{liminf_limsup_equiv}).

The reverse implication $\Longleftarrow$ in (\ref{liminf_limsup_equiv})
follows from the $\Longrightarrow$ case by replacing $\(x_n\)_{n\in\NN}$
by $\(-x_n\)_{n\in\NN}$ in (\ref{eq_curr_x_n}) that leads to a sign
change for $\kappa_n$.
\end{proof}

If we expect solutions of (\ref{eq_curr_x_n}) to behave well as
$n\to\infty$ then it is natural to assume that so do the coefficients.
This is expressed in the following statement.

\begin{theorem} \label{theorem_asy1}

Let the coefficients of {\rm(\ref{eq_curr_x_n})} be such that the
following four limits satisfy
\begin{equation} \label{coeff_limit1}
	\sigma_{0} \DEF \lim_{n\to\infty} \sigma_{n,0}  > 0
	\quad \& \quad
	p_{\pm 1} \DEF \lim_{n\to\infty} \, \sigma_{n,\pm 1} \sqrt { \frac {\ell_{n \pm 1}} {\ell_n} } \in \RR
	\quad \& \quad
	q \DEF \lim_{n\to\infty} \frac {\kappa_n} {\sqrt{\ell_n}} \in \RR \, .
\end{equation}
If (the not necessarily positive) $\bX = \(x_n\)_{n\in\NN}$ satisfies
{\rm(\ref{eq_curr_x_n})} and
\begin{equation} \label{cond_liminf2}
	\liminf_{n\to\infty} \frac {x_n} {\sqrt{\ell_n}} \ge 0 ,
\end{equation}
then $\lim_{n\to\infty} {x_n} / {\sqrt{\ell_n}}$ exists and
\begin{equation} \label{x_n_lim1}
	\lim_{n\to\infty} \frac {x_n} {\sqrt{\ell_n}} = 
	\frac 
		{-q + \sqrt{q^2 + 4 \( p_{1} + \sigma_0 + p_{-1} \)}} 
		{2 \( p_{1} + \sigma_0 + p_{-1} \)}
	 \, ,
\end{equation}
and if (the not necessarily negative) $\bX = \(x_n\)_{n\in\NN}$
satisfies {\rm(\ref{eq_curr_x_n})} and
\begin{equation} \label{cond_limsup2}
	\limsup_{n\to\infty} \frac {x_n} {\sqrt{\ell_n}} \le 0 ,
\end{equation}
then again $\lim_{n\to\infty} {x_n} / {\sqrt{\ell_n}}$ exists and
\begin{equation} \label{x_n_lim2}
	\lim_{n\to\infty} \frac {x_n} {\sqrt{\ell_n}} = 
	- \frac 
		{q + \sqrt{q^2 + 4 \( p_{1} + \sigma_0 + p_{-1} \)}} 
		{2 \( p_{1} + \sigma_0 + p_{-1} \)}
	 \, .
\end{equation}

\end{theorem}

\begin{corollary} \label{theorem_asy2}

Let the coefficients of {\rm(\ref{eq_curr_x_n})} be such that the
following five limits satisfy
\begin{equation} \notag
	\sigma_{1} \DEF \lim_{n\to\infty} \sigma_{n,1} \ge 0
	\quad \& \quad
	\sigma_{0} \DEF \lim_{n\to\infty} \sigma_{n,0} > 0
	\quad \& \quad
	\sigma_{-1} \DEF \lim_{n\to\infty} \sigma_{n,-1} \ge 0
\end{equation}
and
\begin{equation} \notag
	p \DEF \lim_{n\to\infty} \sqrt { \frac {\ell_{n+1}} {\ell_n} } > 0
	\quad \& \quad
	q \DEF \lim_{n\to\infty} \frac {\kappa_n} {\sqrt{\ell_n}} \in \RR \, ,
\end{equation}
where all five limits are finite. If (the not necessarily positive) $\bX
= \(x_n\)_{n\in\NN}$ satisfies {\rm(\ref{eq_curr_x_n})} and
{\rm(\ref{cond_liminf2})}, then $\lim_{n\to\infty} {x_n} /
{\sqrt{\ell_n}}$ exists and
\begin{equation} \notag
	\lim_{n\to\infty} \frac {x_n} {\sqrt{\ell_n}} = 
	\frac 
		{-q + \sqrt{q^2 + 4 \( \sigma_{1} \, p + \sigma_0 + \sigma_{-1} \, p^{-1} \)}} 
		{2 \( \sigma_{1} \, p + \sigma_0 + \sigma_{-1} \, p^{-1} \)}
	 \, ,
\end{equation}
and if (the not necessarily negative) $\bX = \(x_n\)_{n\in\NN}$
satisfies {\rm(\ref{eq_curr_x_n})} and {\rm(\ref{cond_liminf2})}, then
again $\lim_{n\to\infty} {x_n} / {\sqrt{\ell_n}}$ exists and
\begin{equation} \notag
	\lim_{n\to\infty} \frac {x_n} {\sqrt{\ell_n}} = 
	- \frac 
		{q + \sqrt{q^2 + 4 \( \sigma_{1} \, p + \sigma_0 + \sigma_{-1} \, p^{-1} \)}} 
		{2 \( \sigma_{1} \, p + \sigma_0 + \sigma_{-1} \, p^{-1} \)}
	\, .
\end{equation}

\end{corollary}

\begin{proof}[Proof of Theorem~\ref{theorem_asy1}]

First we deal with the case (\ref{cond_liminf2}). We will use the
``$t_n$'' notation, see (\ref{x_n_2_t_n}). By
Theorem~\ref{theorem_bound1},
\begin{equation} \notag
	0 \le 
	\ell \DEF \liminf_{n\to\infty} t_n
	\le 
	L \DEF \limsup_{n\to\infty} t_n
	< \infty \, ,
\end{equation}
and our goal is to show that $L \le \ell$. Once this is done, then
evaluating $\lim_{n\to\infty} t_n$ is straightforward. Our approach is
based on the \emph{Freud Kunstgriff\/}.

Use (\ref{cond_liminf2}) to find a nonincreasing sequence of nonnegative
numbers $\(\varepsilon_n\)_{n\in\NN}$ such that\footnote{If $\bX$ is
nonnegative, then just set $\varepsilon_n = 0$ for $n\in\NN$.}
\begin{equation} \notag
	\lim_{n\to\infty} \varepsilon_n = 0 
	\quad \& \quad
	t_n + \varepsilon_n \ge 0 \, , \quad \forall n\in\NN .
\end{equation}
We rewrite equation (\ref{eq_curr_t_n}) as
\begin{align} \label{eq_curr_t_n2}
	& \phantom{\hskip 2.00 cm} 
	1 =
		\( t_n + \varepsilon_n \) \times
	\\
		&\( 
		\sigma_{n,1} \, \sqrt {\frac {\ell_{n+1}} {\ell_n}} \, \( t_{n+1} + \varepsilon_{n+1} \) + 
		\sigma_{n,0} \, \( t_n + \varepsilon_n \) + 
		\sigma_{n,-1} \, \sqrt {\frac {\ell_{n-1}} {\ell_n}} \, \( t_{n-1} + \varepsilon_{n-1} \)
		\) + \notag
	\\
		&\phantom{\hskip 2.00 cm}
		\frac {\kappa_n} {\sqrt{\ell_n}} \, \( t_n + \varepsilon_n \) +
		R_n
		\, , \notag
\end{align}
where
\begin{align} \notag
	& \phantom{\hskip 2.00 cm} 
	R_n \DEF 
	-t_n \(
	\sigma_{n,1} \, \sqrt {\frac {\ell_{n+1}} {\ell_n}} \, \varepsilon_{n+1} +
	\sigma_{n,0} \, \varepsilon_n +
	\sigma_{n,-1} \, \sqrt {\frac {\ell_{n-1}} {\ell_n}} \, \varepsilon_{n-1}
	\) -
	\\
	& \varepsilon_n \( 
	\sigma_{n,1} \, \sqrt {\frac {\ell_{n+1}} {\ell_n}} \, \( t_{n+1} + \varepsilon_{n+1} \) + 
	\sigma_{n,0} \, \( t_n + \varepsilon_n \) + 
	\sigma_{n,-1} \, \sqrt {\frac {\ell_{n-1}} {\ell_n}} \, \( t_{n-1} + \varepsilon_{n-1} \)
	\) + \notag
	\\
	& \phantom{\hskip 2.00 cm} 
	-\frac {\kappa_n} {\sqrt{\ell_n}} \, \varepsilon_n \, . \notag
\end{align}
Clearly,
\begin{equation} \notag
	\lim_{n\to\infty} R_n = 0
\end{equation}
because each $\varepsilon$-term in it goes to $0$ as $n\to\infty$ and
every coefficient of every $\varepsilon$-term in it is $\cO(1)$.
The advantage of (\ref{eq_curr_t_n2}) is that, except for $R_n$, every
$t+\varepsilon$-term is nonnegative in it and that every coefficient of
every such $t+\varepsilon$-term is both nonnegative and convergent.

Now we are in the position to use the \emph{Freud Kunstgriff\/}.
Pick $\NN_\ell \subset \NN$ and $\NN_L \subset \NN$ such that
\begin{equation} \notag
	\lim_{\stackrel {n\to\infty} {n \in \NN_\ell}} t_n = \ell
	\qquad \& \qquad
	\lim_{\stackrel {n\to\infty} {n \in \NN_L}} t_n = L ,
\end{equation}
and let $n\to\infty$ first over $\NN_\ell$ and then over $\NN_L$ in
(\ref{eq_curr_t_n2}). We get
\begin{equation} \notag
	1 \le \ell \( p_{1} \, L + \sigma_0 \, \ell + p_{-1} \, L \) + q \, \ell
\end{equation}
and 
\begin{equation} \notag
	1 \ge L \( p_{1} \, \ell + \sigma_0 \, L + p_{-1} \, \ell \) + q \, L ,
\end{equation}
respectively, from which
\begin{equation} \notag
	L \( p_{1} \, \ell + \sigma_0 \, L + p_{-1} \, \ell \) + q \, L
	\le 
	\ell \( p_{1} \, L + \sigma_0 \, \ell + p_{-1} \, L \) + q \, \ell \, ,
\end{equation}
that is,  $L \le \ell$ so that $\ell = L$. Once we know that $T \DEF
\lim_{n\to\infty} t_n$ exists, we just let $n\to\infty$ either in
(\ref{eq_curr_t_n2}) or in (\ref{eq_curr_t_n}) to obtain
\begin{equation} \notag
	1 = T \( p_{1} \, T + \sigma_0 \, T + p_{-1} \, T \) + q \, T
\end{equation}
and the positive solution of
\begin{equation} \notag
	\( p_{1} + \sigma_0 + p_{-1} \, \) T^2 + q \, T - 1 = 0
\end{equation}
yields (\ref{x_n_lim1}).

If, instead of (\ref{cond_liminf2}), condition (\ref{cond_limsup2}) holds,
then, as observed in the proof of Theorem~\ref{theorem_bound1},
replacing $\(x_n\)_{n\in\NN}$ by $\(-x_n\)_{n\in\NN}$ in
(\ref{eq_curr_x_n}) leads to a sign change for $\kappa_n$ and that
results in a sign change of $q$ in (\ref{coeff_limit1}) so that 
(\ref{x_n_lim2}) follows from (\ref{x_n_lim1}).
\end{proof}

\begin{note}

It remains to be seen if conditions (\ref{cond_liminf2})
and (\ref{cond_limsup2}) in Theorem~\ref{theorem_asy1} can be replaced
by a one-sided $\cO(1)$ condition similarly as it is done in
Theorem~\ref{theorem_bound1}.

\end{note}

\section{Acknowledgments}
\label{sec_ack}


\textrm{We thank Vilmos Totik and those participants of the 2003 Schweitzer
competition whose solutions of Problem~\#6 we had the privilege to study and
to adopt, especially P\'eter Varj\'u, see the details at the end of
\S\ref{sec_intro}. We also thank the referees whose suggestions helped to
improve the presentation.}



\end{document}